\documentclass[10pt,a4paper]{amsart} 

\usepackage[british,american]{babel}

\usepackage{blindtext}
 
\usepackage{cite}

\usepackage[draft]{optional}

\usepackage[margin=0cm]{geometry}
\usepackage{lipsum}

\usepackage{mathrsfs}
\usepackage{amsfonts, amsmath, wasysym, caption}

\usepackage{amssymb,amsthm, paralist}

\usepackage{amsopn}
\usepackage{accents,bbm,bibgerm,dsfont,eucal,esint,paralist,url,verbatim,wasysym}
\usepackage[normalem]{ulem}
\usepackage{bbold}

\usepackage{
latexsym,
nicefrac,
rotating
}

\usepackage[usenames]{color}
\usepackage{tikz}
\usetikzlibrary{decorations.pathreplacing}
\usepackage{url}

\definecolor{darkgreen}{rgb}{0,0.5,0}
\definecolor{darkred}{rgb}{0.7,0,0}

\usepackage[colorlinks, 
citecolor=darkgreen, linkcolor=darkred
]{hyperref}

\usepackage{esint}
\usepackage{bibgerm}
\usepackage[normalem]{ulem}

\usepackage{enumitem}

\theoremstyle{plain}
\newtheorem{lemma}{Lemma}[section]
\newtheorem{thm}[lemma]{Theorem}
\newtheorem{prop}[lemma]{Proposition}
\newtheorem{cor}[lemma]{Corollary}
\theoremstyle{definition}

\newtheorem{rmk}[lemma]{Remark}
\newtheorem{rem}[lemma]{Remark}

\numberwithin{equation}{section}

\newcommand{\al}{\alpha}

\newcommand{\ga}{\gamma}

\newcommand{\de}{\delta}

\newcommand{\la}{\lambda}

\newcommand{\si}{\sigma}

\newcommand{\vph}{\varphi}

\newcommand{\ep}{\varepsilon}

\newcommand{\R}{\ensuremath{{\mathbb R}}}
\newcommand{\N}{\ensuremath{{\mathbb N}}}

\newcommand{\Z}{\ensuremath{{\mathbb Z}}}

\newcommand{\Hyp}{\ensuremath{{\mathbb H}}}

\newcommand{\weakto}{\rightharpoonup}

\newcommand{\grad}{\nabla}

\newcommand{\norm}[1]{\Vert#1\Vert}

\newcommand{\brmk}{\begin{rmk}}
\newcommand{\ermk}{\end{rmk}}
\newcommand{\partref}[1]{\hbox{(\csname @roman\endcsname{\ref{#1}})}}

\newcommand{\beq}{\begin{equation}}
\newcommand{\eeq}{\end{equation}}
\newcommand{\beqs}{\begin{equation*}}
\newcommand{\eeqs}{\end{equation*}}
\newcommand{\beqa}{\begin{equation}\begin{aligned}}
\newcommand{\eeqa}{\end{aligned}\end{equation}}
\newcommand{\beqas}{\begin{equation*}\begin{aligned}}
\newcommand{\eeqas}{\end{aligned}\end{equation*}}

\newcommand{\half}{\frac{1}{2}}

\newcommand{\M}{\ensuremath{{\mathcal M}}}

\newcommand{\abs}[1]{{\lvert#1\rvert} }
\newcommand{\babs}[1]{{\big\lvert#1\big\rvert} } 
\newcommand{\eps}{\varepsilon}
\newcommand{\na}{\nabla}

\newcommand{\Hol}{{\mathcal{H}}} 

\newcommand{\ddeps}{\tfrac{d}{d\eps}\vert_{\eps=0}}

\newcommand{\lan}{\langle}
\newcommand{\ran}{\rangle}

\newcommand{\Rea}{\text{\textnormal{Re}}}

\newcommand{\dvg}{dv_g}
\newcommand{\ubar}{{\bar u}}
\newcommand{\gbar}{{\bar g}}

 \newcommand{\tr}{\mathrm{tr}}
 
\newcommand{\ReP}{\mathrm{Re}}
\newcommand{\partiald}[2]{\frac{\partial #1}{\partial #2}}
\newcommand{\inj}{\mathrm{inj}}

\newcommand{\teich}{Teichm\"uller }

\newcommand{\id}{\mathrm{id}}
\newcommand{\Eid}{\mathcal{E}}
\newcommand{\Mab}{\mathcal{M}^*}
\newcommand{\arcosh}{\mathrm{arcosh}}
\newcommand{\loj}{\L{}ojasiewicz }
\newcommand{\gmu}{{\hat g(\mu)}}
\newcommand{\graduF}{\na_u F} 
\newcommand{\gradmuF}{\na_\mu F} 
\newcommand{\dwp}{d_{\mathrm{WP}}}

\title[Uniqueness and nonuniqueness of limits of Teichm\"uller harmonic map flow]
{{Uniqueness and nonuniqueness of limits of\\ Teichm\"uller harmonic map flow}
\\ 
}
\author{James Kohout, Melanie Rupflin and Peter M. Topping}
\date{\today}

\begin{document}

\maketitle


\begin{abstract}
The harmonic map energy of a map from 
a closed, constant-curvature surface to a closed target manifold can be seen as a functional on the space of maps and domain metrics.
We consider the gradient flow for this energy. 
In the absence of singularities, previous theory established that the flow converges to a branched minimal immersion, but only at a sequence of times converging to infinity, and only after pulling back by a sequence of diffeomorphisms.
In this paper we investigate whether it is necessary to pull back by these diffeomorphisms, and whether the convergence is uniform as $t\to\infty$.
\end{abstract}

\section{Introduction and results}

In their seminal work \cite{ES1964} from 1964, Eells and Sampson showed that the harmonic map flow into a closed manifold of nonpositive sectional curvature admits global solutions starting with arbitrary smooth initial data, which converge to a limiting harmonic map when restricted to a suitable sequence of times $t_i\to\infty$. In practice, it is important to understand whether the flow converges as $t\to\infty$, rather than merely at a sequence of times $t_i\to\infty$, for example when one considers a smooth family of initial maps and would like to extract a smooth family of subsequent flows and corresponding \emph{unique} limits.
Soon after the work of Eells and Sampson, Hartman \cite{Hartman1967} exploited the nonpositive curvature of the target in order to establish the required uniqueness in this case. For harmonic map flow into more general targets, the uniqueness turned out to be false in general \cite{Topping1996}, \cite{Topping1997}. The flow can limit to a nontrivial circle of harmonic maps.

Meanwhile, the question of uniqueness of blow-ups of singularities has become a central problem across geometric analysis. Early instances are the questions of uniqueness of tangent cones for minimal surfaces, and of tangent maps for harmonic maps; these are key issues in the understanding of the regularity theory for the respective variational problems, and the topic remains a subject of intense study. The first example of nonuniqueness was given by White in 1992 \cite{White1992} in the context of harmonic maps, where a suitable warped product target metric was constructed that is necessarily not real analytic by the work of Simon \cite{Simon1983}. 
Following Struwe's initiation of a theory of global weak solutions of the harmonic map flow from two-dimensional domains, perhaps developing bubbles as singularities \cite{Struwe1985}, 
the question of uniqueness of blow-ups of singularities came to the fore also in this context.
Nonuniqueness was demonstrated in \cite{Topping04} by constructing examples with different finite-time bubbles arising for different sequences of blow-ups.
Even in the case of convergence of the flow at infinite time, it is apparent from \cite{Topping1996} that for certain flows one could have one sequence of times $t_i\to\infty$ at which the flow converges smoothly to a harmonic map, and a different sequence of times $s_j\to\infty$ at which the flow undergoes bubbling. 

An even richer set of uniqueness questions arises when one replaces the harmonic map flow from surfaces by the alternative gradient flow for the harmonic map energy in which one evolves not just the map, but also the domain metric, as introduced in \cite{DLL2006} for maps from tori and in \cite{RT2016} for maps from closed surfaces of general type, in order to find minimal immersions. The purpose of this paper is to settle the majority of these questions by constructing several similar \emph{winding} examples of the flow into smooth targets as well as proving a convergence result for analytic targets. In order to understand these questions, we recall some of the basic theory for this alternative \emph{Teichm\"uller} harmonic map flow.

Given a smooth closed oriented surface $M$ of genus $\gamma \geq 1$ and a smooth closed Riemannian manifold $(N,g_N)$ the flow can be thought of as the gradient flow for the energy
\begin{equation}\label{eqn:dirichlet-energy}
	E(u,g) = \half\int_M |du|^2_g \dvg
\end{equation}
with respect to both the map $u \colon M \to N$ and a metric $g \in \M_c$, where $\M_c$ is the space of (smooth) constant curvature $c\in\{-1,0\}$ metrics on $M$ (with the restriction that they have unit area when $c=0$). The flow equations are given by
\begin{equation}\label{eqn:flow}
\partiald{u}{t} = \tau_g(u); \qquad
\partiald{g}{t} = \frac{\eta^2}{4}\ReP(P_g(\Phi(u,g))),
\end{equation}
where $\tau_g(u)$ denotes the tension field of $u$ (i.e. $\tr\nabla du$), $P_g$ denotes the $L^2$ orthogonal projection from the space of quadratic differentials onto the finite-dimensional space of \emph{holomorphic} quadratic differentials, $\Phi(u,g)$ denotes the Hopf differential and $\eta > 0$ is a fixed parameter. The flow decreases the energy $E(t) := E(u(t),g(t))$ according to
\begin{equation} \label{eqn:energy-decay}
\frac{dE}{dt} = -\norm{\tau_g(u)}_{L^2(M,g)}^2 - \frac{\eta^2}{32}\norm{P_g(\Phi(u,g))}_{L^2(M,g)}^2.
\end{equation}
In particular, we will use that the evolution of $g$ is constrained by
\begin{equation}
\label{cherry}
\int_{t_0}^{t_1}\|\partial_t g\|_{L^2(M,g(t))}dt
\leq C(\eta)\int_{t_0}^{t_1}\left(-\frac{dE}{dt}\right)^\half dt
\leq C(\eta)[E(0)]^\half [t_1-t_0]^\half.
\end{equation}
For further details on these objects and the flow see \cite{RT2016} and \cite{RT2019}. 

Given any initial data $(u_0,g_0) \in H^1(M,N)\times \M_c$ we know that a global weak solution of the flow \eqref{eqn:flow} exists which is smooth except at possibly finitely many times,
and is unique amongst weak energy nonincreasing flows.
These singularities are caused by the bubbling off of finitely many harmonic spheres, a change of topology of the domain when $\inj(M,g(t))\to 0$ as the time $t$ approaches a singular time (only if the genus $\gamma \geq 2$) or both of these phenomena happening at the same time: see \cite{Rexistence, RT2019}. When the genus $\gamma = 1$, where the flow was first studied by \cite{DLL2006}, finite time singularities of the metric are excluded as the completeness of Teichm\"uller space gives an a priori positive lower bound on the injectivity radius on any finite time interval.

In situations where the metric does not degenerate as $t\to \infty$, 
the results of 
\cite{RT2016} (and \cite{DLL2006} when $\gamma =1$) 
guarantee the existence of a sequence of times $t_i\to\infty$, a sequence of orientation preserving diffeomorphisms $\psi_i \colon M \to M$, a metric $\bar{g} \in \M_c$ and a weakly conformal harmonic map $\bar{u} \colon (M,\bar{g}) \to N$ for which we have $\psi_i^*g(t_i) \to \bar{g}$ smoothly and $u(t_i)\circ \psi_i \rightharpoonup \bar{u}$ weakly in $H^1(M,N)$ and strongly away from finitely many points. In this paper we investigate questions related to this result. Firstly, can we have winding behaviour as in the harmonic map flow. That is, does the limit depend on the subsequence taken above? Secondly, 
do we actually need to pull back by diffeomorphisms, or do we have convergence anyway?
Finally, what happens if the target is required to be analytic?

In the first part of this paper we construct smooth (but not analytic) settings where 
winding behaviour of the domain metric does indeed occur even for initial data for which 
the map component remains smooth for all time, and at infinite time, and for which we have a uniform lower bound on the injectivity radius of the domain. For this part of the paper we consider the flow on tori, and recall that in this case 
the analysis of the flow  is simplified by the fact that 
not only does the velocity $\partiald{g}{t}$ lie 
in a two-dimensional subspace of the infinite dimensional space $T\M_0$, but also the distribution defined by these ``horizontal'' subspaces is integrable. Indeed, by pulling back the initial data and the whole flow by a fixed diffeomorphism, it suffices to consider the flow of metrics in the explicit two-parameter family of flat unit area metrics 
\begin{equation}\label{def:Mab} 
\Mab :=\{g_{a,b}= T_{a,b}^*g_E: (a,b)\in \Hyp\}
\end{equation}
on $T^2:=\R^2/\Z^2$,
where $T_{a,b} \colon \R^2 \to \R^2$ is the linear map sending $(1,0) \mapsto \frac{1}{\sqrt{b}}(1,0)$ and $(0,1) \mapsto \frac{1}{\sqrt{b}}(a,b)$, $g_E$ is the Euclidean metric and $\Hyp$ is the upper-half plane. 
The Weil-Petersson metric on $\Mab$ then corresponds, up to a scaling factor, to the hyperbolic distance on $\Hyp$.

As we shall see, we can construct targets for which  \teich harmonic map flow from the torus $T^2$ exhibits winding behaviour, has nonunique limits and does not converge for any sequence of times $t_i\to \infty$ unless we pull back by a sequence of diffeomorphisms. We can arrange that this sequence of diffeomorphisms must diverge in the mapping class group. 

Indeed, as we shall prove, we can construct smooth targets so that the flow of the metrics $g(t)$ has a prescribed asymptotic behaviour as described in our first main result: 

\begin{thm} \label{thm:main}
Let $(G_s)_{s\in[0,\infty)}$ be any smooth curve in $\Mab$ whose projection to moduli space is 1-periodic, i.e. for which there exists a diffeomorphism $\varphi \colon T^2\to T^2$ so that $G_s =\varphi^*G_{s+1}$ for every $s\in [0,\infty)$.

Then there exists a smooth closed target manifold $(N,g_N)$ and initial data $(u_0,g_0)\in C^\infty(T^2,N)\times \Mab$ such that the corresponding solution $(u(t),g(t))$ of \teich harmonic map flow has ${\sup_{t} \norm{\na u(t)}_{L^\infty(M,g(t))}<\infty}$, 
and has the following asymptotic behaviour: 

\begin{itemize}
	\item There exists a  smooth function $z \colon [0,\infty) \to [0,\infty)$ satisfying $z(t) \to \infty$ as $t \to \infty$ so that 
 	\begin{equation} \label{claim:main-thm-metric}
 		\dwp(g(t), G_{z(t)}) \to 0 \text{ as } t \to \infty
 	\end{equation}
 	where $\dwp$ is the Weil-Petersson distance on $\Mab$.
	\item For every $z\in [0,1)$ there exists a sequence of times $t^z_i\to \infty$ so that after pulling back by the $i^{th}$ iterate 
	$\varphi^i$ the maps $u(t^z_i)$ converge smoothly,
	\begin{equation} \label{claim:main-thm-map} 
		u(t^z_i)\circ \varphi^{i}\to u_{z},
	\end{equation} 
to a minimal immersion $u_{z}:T^2 \to N$. The resulting minimal immersions $u_{z_1}$ and $u_{z_2}$ with $z_1 \neq z_2$ parametrise different minimal surfaces in $(N,g_N)$. After pulling back by $\varphi^i$ the metrics $g(t^z_i)$  also converge
	\begin{equation*}
		(\varphi^i)^*g(t^z_i) \to G_z \text{ in } C^\infty.
	\end{equation*}
\end{itemize}
\end{thm}

Since $\Mab$ is a horizontal submanifold of the space of metrics, any diffeomorphism will pull it back to another horizontal submanifold. The specific diffeomorphism $\vph$ pulls back $G_1\in\Mab$ to $G_0\in\Mab$, and $\Mab$ is connected and complete, 
so $\vph$ must pull back $\Mab$ to $\Mab$ itself.
In particular, given a solution of the flow as above, we know that $\vph^*g(t)\in\Mab$ for each $t\in [0,\infty)$.

Theorem \ref{thm:main} immediately yields the following corollary about the possible behaviour of Teichm\"uller harmonic map flow into suitable smooth closed targets:

\begin{cor}
	The limit of solutions of \teich harmonic map flow as $t\to \infty$ can be non-unique, even after pull-back by diffeomorphisms. 
\end{cor}

%

The diffeomorphisms $\varphi^i$ that we use to pull back in \eqref{claim:main-thm-map} are iterated compositions of the diffeomorphism $\varphi$ obtained from $G_s$. 
One possibility that we consider is that $G_s$ is a periodic curve, as is the case when $\varphi$ is the identity, for example.
A quite different situation arises in the case that $G_s$ leaves any compact subset of \teich space. In particular, $\varphi$ can represent a Dehn twist.

In this second case we obtain the following:

\begin{cor} \label{cor:metric-diverges}
	There exist a closed target manifold and a solution of \teich harmonic map flow from $T^2$ into that target whose metric component leaves any compact subset of Teichm\"uller space even though the injectivity radius of the domain remains bounded away from zero and hence the metric component remains in a compact subset of moduli space. 
\end{cor}
When the curve $G_s$ is a non-trivial closed curve we obtain the following:
\begin{cor}
	There exist a closed target manifold and a solution
of \teich harmonic map flow  from $T^2$ into that target for which the metric component stays in a compact subset of \teich space but for which the limit as $t\to \infty$ is not unique. In particular we can choose sequences of times $t_i,\tilde{t}_i \to \infty$ so that $(u(t_i),g(t_i))$ and $(u(\tilde{t}_i),g(\tilde{t}_i))$ converge (without having to pull back by diffeomorphisms) to limits which parametrize different minimal surfaces. 
\end{cor}

While the above results show that solutions of Teichm\"uller harmonic map flow into smooth closed targets can exhibit winding behaviour even in situations where the energy density is bounded uniformly from above and the injectivity radius from below, our second main result excludes this behaviour for analytic targets. 
Recall that one starts the analysis of the asymptotics of the flow by selecting a sequence of times $t_j\to\infty$ with 
$\norm{\partial_t(u,g)(t_j)}_{L^2(M,g(t_j))}\to 0$
and then analyses the bubbling and/or collar degeneration singularities that might occur in the limit \cite{RT2016,RTZ,HRT}.
If there exists any such sequence for which no singularities occur, when the target is analytic, then we obtain uniform convergence:
\begin{thm}\label{thm:analytic}
Let $(N,g_N)$ be a closed analytic manifold of any dimension and let $M$ be a closed oriented surface of genus $\ga\geq 1$. Let $(u,g)$ be any global weak solution of Teichm\"uller harmonic map flow \eqref{eqn:flow},
with nonincreasing energy, for which there is a sequence $t_j \to \infty$ such that
\beq
\label{ass:map-metric-converg}
\lim_{j\to\infty} \norm{\partial_t(u,g)(t_j)}_{L^2(M,g(t_j))}=0, \quad \sup_{j} \norm{\na u(t_j)}_{L^\infty(M,g(t_j))}<\infty \text{ and } \inf_{j}\inj(M,g(t_j))>0.\eeq
Then 
$(u,g)(t)$ converges smoothly as $t\to \infty$ to a limiting pair $(u_\infty,g_\infty)$ consisting of a metric $g_\infty\in \M_c$ and a weakly conformal harmonic map ${u_\infty \colon (M,g_\infty)\to (N,g_N)}$.
\end{thm}

We are making the restriction that $\ga\geq 1$ in this paper. One can make sense of 
the flow in the case $\ga=0$, but it reduces to the classical harmonic map flow and the theorem above then already follows from \cite{Simon1983}.

\begin{rem}\label{rem:weaker-energy-ass}
The assumption in \eqref{ass:map-metric-converg} can be weakened. It is  enough to ask that there exist $r_1>0$ and $t_j\to \infty$ so that $\inf_{j}\inj(M,g(t_j))>0$ and so that we have uniform control on the energy on balls of radius $r_1$ of
\beq
\label{est:energy-small}
\half \int_{B^{g(t_j)}_{r_1}(x)} \abs{du(t_j)}^2_{g(t_j)} dv_{g(t_j)}\leq \eps_0 \text{ for every } x\in M \text{ and } j\in \N
\eeq
where $\eps_0=\eps_0(N,g_N)>0$, since parabolic regularity theory and \eqref{eqn:energy-decay} would allow us to choose nearby times for which \eqref{ass:map-metric-converg} holds, as can be 
derived from the proof of Theorem \ref{thm:analytic}.
\end{rem}

The structure of the paper is as follows. In Section \ref{example_sect} we prove Theorem \ref{thm:main} in two steps, first constructing an auxiliary non-compact target
for which the map component drifts off to infinity while the metric component closely follows the prescribed curve $G_s$ and then in a second step  
constructing the desired closed target, which will contain the auxiliary target as a totally geodesic submanifold, for which the flow exhibits the claimed winding behaviour. 
The proof of our second main result  Theorem \ref{thm:analytic} on the asymptotic convergence of the flow into analytic targets is then carried out in Section \ref{sect:3} and is based on a \L{}ojasiewicz-Simon inequality for the Dirichlet energy on the set of pairs of maps and metrics that we state in Theorem \ref{thm:loj-ineq}.

\section{Proof of Theorem \ref{thm:main}}
\label{example_sect}

Our set-up is inspired by the construction of a target by the third author in \cite{Topping04} for which solutions of harmonic map flow exhibit winding behaviour. Our first step here will be to construct a non-compact target manifold $N_0$ for which the metric component of the flow \eqref{eqn:flow} closely follows the prescribed curve of metrics $G_s$ as the image of the map drifts off to infinity. We will then wrap this non-compact target around a cylinder to obtain a smooth, closed target for which solutions wind around the cylinder and approach a circle of critical points. 

The first, non-compact, target is given by $N_0 = \R\times T^2$ equipped with the metric
\begin{equation} \label{def:line-cross-torus-metric}
	g_{N_0} = dz^2 + f_0(z)G_z
\end{equation}
where $G_s$ is the prescribed curve of metrics in $\Mab$ (extended to $(-\infty,0]$ via $G_s=\varphi^*G_{s+1}$) 
 and $f_0 \in C^\infty (\R, [1,\infty))$ is bounded with $-C\leq f_0'<0$ and $\lim_{z\to \infty} f_0(z)=1$. 

We will consider maps $u\colon T^2\to N_0$ whose first component is constant in space and whose second component is given by the identity map of the torus. The energy of such maps $u=(z,\id)$ is given by 
\begin{equation}\label{eqn:energy-structure-line}
	E(u, g) = f_0(z) \Eid(g,G_z) \text{ where } \Eid(g,G):= E(\id \colon (T^2,g) \to (T^2,G))
\end{equation}
for $g, G\in \Mab$ and we always have 
$\Eid(g,G)\geq \text{Area}(T^2,G)=1$, with equality if and only if $\id: (T^2,g)\to (T^2,G)$ is conformal. As metrics in $\Mab$ are unique in their conformal class we thus have 
\beq 
\label{est:energy-equal-f}
 E((z,\id),g)\geq f_0(z) \text{ with equality  if and only if } g=G_z.\eeq
Since the identity map $\id \colon (T^2,g) \to (T^2,G)$ is harmonic for any metrics $g, G \in \Mab$, we have that for each $z\in\R$ the $T^2$ component of the tension of any map 
$u=(z,\id)$ from $(T^2,g)$ to $(N_0,g_{N_0})$ vanishes, while the first component of the tension is given by 
\begin{equation} \label{eqn:line-cross-torus-tension}
	\tau_g(u)^\R = -f_0(z) D_{G}\Eid_{(g,G_z)}(\tfrac{d G_z}{dz} ) - f_0'(z)\Eid(g,G_z).
\end{equation}

\begin{rem} \label{rem:no-crit-point}
If $\bar g\in\Mab$ and $u=(z,\id)\colon (T^2,\bar g)\to (N_0,g_{N_0})$ is conformal, 
then we have $\bar g=G_{z}$. 
Hence the first term on the right-hand side of  \eqref{eqn:line-cross-torus-tension} vanishes while the second term is always non-zero as $f_0'\neq 0$. We thus observe that there are no maps $u$ of this form that are both harmonic and conformal. 
\end{rem}

Although our target $N_0$ is non-compact, and the existing theory for this flow is phrased for compact targets, the short-time existence
for an initial map of the form $u_0=(z_0,\id)$ and any initial metric $g_0\in \Mab$ will follow from simple ODE theory.
Recall from \cite{DLL2006} that the equation for the metric component $g(t) = g_{(a,b)(t)}$ reduces to a system of ODEs for $(a,b) \in \Hyp$. Standard ODE theory gives short-time existence of a unique solution to the system consisting of the ODEs for $(a,b)(t)$ coupled with the equation for $z(t)$
\begin{equation} \label{eqn:line-map-ode}
\dot{z}  = \tau_g(u)^\R = -f_0(z) D_{G}\Eid_{(g,G_z)}(\tfrac{d G_z}{dz} ) - f_0'(z)\Eid(g,G_z),
\end{equation}
where $g = g_{(a,b)}$. Therefore setting $u(t) = (z(t),\id)$ and $g(t) = g_{(a,b)(t)}$ we obtain that $(u,g)$ is a solution to the flow.

We will now show that the solution exists for all times and does not degenerate, as well as control the Weil-Petersson distance between the metric component of the flow and the prescribed curve of metrics $G_s$.

\begin{lemma} \label{lem:line-cross-torus-sol}
Let $(N_0,g_{N_0})$ be as in \eqref{def:line-cross-torus-metric} with $G_s$ the prescribed curve of metrics as in Theorem \ref{thm:main}  and let $f_0 \in C^\infty (\R, [1,\infty))$ be bounded with $-C\leq f_0'<0$ and $\lim_{z\to \infty} f_0(z)=1$.

Given any $z_0\in\R$ and any $g_0\in \Mab$ we have that the solution $(u,g)=((z,\id),g)$ of the flow \eqref{eqn:flow} with initial map $u_0=(z_0,\id)$ and initial metric $g_0$ exists and is smooth for all times, 
and does not degenerate at infinite time in the sense that 
	\begin{equation} \label{claim:does-not-deg}
		 \inj(T^2,g(t)) \geq c>0 \text{ for all } t\in [0,\infty).
	\end{equation}
Here $c$ depends only on an upper bound for the initial energy and on the curve $G_s$. Furthermore, the Weil-Petersson distance between the metric component and the given curve of metrics $G_s$ is controlled by
	\begin{equation}\label{claim:metric-dist-bd}
		\dwp (g(t),G_{z(t)}) \leq C (E(t)-1)^{\frac{1}{2}} \text{ for every } t\in [0,\infty),
	\end{equation}
	where $E(t) := E(u(t),g(t))$ and $C>0$ is a universal constant.

\end{lemma}

Remark \ref{rem:no-crit-point} excludes the possibility that 
this flow has stationary points $(\bar u,\bar g)$ of the above form.
This will later allow us to show that 
the map component of the flow drifts off to infinity while the energy tends to $1$ as $t\to \infty$. By the above lemma this then yields that 
 $g(t)$ and $G_{z(t)}$ have the same limiting behaviour.

For the proof of this lemma we use the following elementary property of the energy $\Eid$ of the identity map between different flat unit area tori, a proof of which we include for the convenience of the reader.

\begin{lemma} \label{lem:metric-dist-bd}
	There is a universal constant $C>0$ such that for any $g, G \in \Mab$ we have
	\begin{equation}
		\dwp (g,G) \leq C(\Eid(g,G) - 1)^{\frac{1}{2}}
	\end{equation}
	where $\Eid(g,G) = E(\id \colon (T^2,g) \to (T^2,G))$ and $\dwp $ is the Weil-Petersson distance. 
\end{lemma}

\begin{proof}
	Since $g,G \in \Mab$ we can write $g = g_{a,b}$ and $G = g_{\alpha,\beta}$ for some $(a,b), (\alpha,\beta) \in \Hyp$, as in  \eqref{def:Mab}. The energy of the identity map written in terms of these parameters is given by
	\begin{equation} \label{eqn:energy-id1}
		\Eid(g_{a,b},g_{\alpha,\beta}) = \frac{1}{2}\int_{T^2} \tr_{g_{a,b}}(g_{\alpha,\beta}) dv_{g_{a,b}}= 1 + \frac{1}{2 b \beta}(a - \alpha)^2 + \frac{1}{2 b \beta}(b - \beta)^2.
	\end{equation}
	This formula connects the quantity $\Eid$ to the hyperbolic distance $d_\Hyp((a,b),(\alpha,\beta))$ on the upper half plane $\Hyp$, indeed we find that $\Eid(g_{a,b},g_{\alpha, \beta}) = \cosh(d_\Hyp(g_{a,b},g_{\alpha, \beta}))$. The Weil-Petersson metric on $\Mab$ is a multiple of the hyperbolic distance, $\dwp  = 2d_\Hyp$, and so
	\begin{equation*}
		\dwp (g,G) = 2\arcosh(\Eid) \leq 2\sqrt{2}(\Eid - 1)^\frac{1}{2}
	\end{equation*}
	using that $\arcosh(x) \leq \sqrt{2}(x-1)^\frac{1}{2}$ for all $x > 1$.
\end{proof}

\begin{proof}[Proof of Lemma \ref{lem:line-cross-torus-sol}]
Let $(u,g)$ be a solution of the flow as in the lemma and let $[0,T)$ be the maximal time  interval on which the solution is defined and smooth. 
There can be no finite-time degeneration of the metric component owing to the completeness of the \teich space of the torus $T^2$, see \cite{DLL2006} Corollary 2.3.  
To show that $T=\infty$ it hence remains to exclude the possibility that $\abs{z(t)}$ goes to infinity in finite time, which we shall do at the end of the proof.

Since $G_s$ is smooth and periodic in moduli space, we have a uniform positive lower
  bound on the lengths of homotopically non-trivial closed curves in the target. Since the initial map is incompressible this suffices to apply 
the argument of Remark 4.4 in \cite{DLL2006} also in the case of a non-compact target manifold to conclude that the injectivity radius of $(T^2,g)$ remains bounded from below by a positive constant $c>0$ that only depends on the initial energy and the curve $G_s$.

We also note that
\eqref{claim:metric-dist-bd} holds true on the maximal time interval $[0,T)$ 
as an immediate consequence of Lemma \ref{lem:metric-dist-bd}, the structure of the energy \eqref{eqn:energy-structure-line} and the fact that $f_0 \geq 1$. 

It finally remains to exclude that  $z(t)$ goes to infinity in finite time. 
By \eqref{cherry}, we see that on any finite time interval the metric can only travel a finite distance
with respect to the Weil-Petersson distance. 
Therefore for any finite $T_1\leq T$, we can  use \eqref{claim:metric-dist-bd} to obtain a compact subset $K$ of $\Mab$ so that $g(t)\in K$ and $G_{z(t)}\in K$ for every $t\in [0,T_1)$. 
Hence \eqref{eqn:line-map-ode}, combined with the assumed bounds on $f_0$,  yields that $\abs{\dot z}$ is bounded on $[0,T_1)$ 
 for any finite $T_1 \leq  T$ and thus that the solution cannot blow up in finite time. 
\end{proof}

We now construct a new target which contains the previously constructed $(N_0,g_{N_0})$ as a totally geodesic submanifold. We first consider an auxiliary non-compact target 
defined as 
\begin{equation} \label{def:2D-noncpct-target}
	N_1 = \mathbb{H} \times T^2, \quad g_{N_1} = g_\Hyp + f(x,y)G_{\log y},
\end{equation} 
where $\mathbb{H} = \{ (x,y) \in \R^2 \mid y > 0 \}$ is the upper half plane equipped with the standard hyperbolic metric
 $g_\Hyp = \frac{1}{y^2}(dx^2 + dy^2)$, $G_s$ is the prescribed curve of metrics from Theorem \ref{thm:main} (extended to $(-\infty,0)$ by periodicity) and ${f \colon \mathbb{H} \to [1,\infty)}$ is a smooth function that will be determined below.

We will in particular choose $f$ with the symmetry
$f(ex,ey)=f(x,y)$. As $(x,y) \mapsto (ex,ey)$ is an isometry of the hyperbolic plane and as $G_s$ satisfies $G_s=\varphi^*G_{s+1}$ this ensures that  
\begin{equation} \label{def:quotient-diffeo}
	\Psi\big((x,y),p\big) = \big((ex,ey), \varphi(p)\big)
\end{equation}
is an isometry of $(N_1,g_{N_1})$. 
Denoting by $\Gamma = \langle \Psi \rangle$ the generated group of isometries of $(N_1,g_{N_1})$ 
we
consider the target $(N_1,g_{N_1})/\Gamma$. The final target will be a slight modification of this in order to make it compact.

For appropriate $f$, the function $f_0(z) := f(1,e^z)$ will induce 
an associated manifold $(N_0,g_{N_0})$ as in 
\eqref{def:line-cross-torus-metric}.
We can then see this as a submanifold of $(N_1,g_{N_1})$
by identifying $(z,p) \in \R\times T^2$ with $((1,y),p) \in \Hyp\times T^2$ for $z = \log y$. As $L := \{1\}\times (0,\infty) \subset \Hyp$ is the image of a geodesic in the hyperbolic plane we obtain that the submanifold is totally geodesic provided $\partial_xf(1,y) = 0$ for every $y \in (0,\infty)$.

We therefore choose our coupling function $f$ as follows. Using the embedding and identification described above the function $1+e^{-y/x}$ leads to a function $f_0$ to which Lemma \ref{lem:line-cross-torus-sol} applies. We will modify the function $f$ so that the $x$-derivative vanishes on the line $L$, without changing the corresponding function $f_0$. Choose a function $\rho \in C^\infty(\R)$ with $\rho(n)=n$ for each $n\in\Z$, that is constant in a neighbourhood of each integer and satisfies $\rho(s+1) = \rho(s) + 1$; we think of $\rho$ as an approximation of the function $s\mapsto s$. Now we define $f:\Hyp\to [1,\infty)$ by 
\begin{equation} \label{def:coupling-function}
f(x,y) = 
\begin{cases}
1 + \exp(-ye^{-\rho(\log x)})  &\text{ if } x > 0\\
1 &\text{ if } x \leq 0.
\end{cases}
\end{equation}
The properties of $\rho$ imply that $f$ has the invariance $f(ex,ey) = f(x,y)$. A computation of the derivatives of $f$ shows that it is smooth and satisfies $\partial_x f(1,y) = 0$ for every $y \in (0,\infty)$. 

Given initial data ${y_0 \in (0,\infty)}$
and $g_0 \in \Mab$ we set 
$z_0 = \log y_0$ and consider the corresponding solution of the flow into $(N_0,g_{N_0})$ defined by \eqref{def:line-cross-torus-metric} for $f _0(z) := f(1,e^z)$. 
As $(N_0,g_{N_0})$ is a totally geodesic submanifold of $(N_1,g_{N_1})$ this induces a solution of the flow with target $(N_1,g_{N_1})$. 
This solution can be projected down to a solution of the flow with target $(N_1,g_{N_1})/\Gamma$ and initial data $(u_0,g_0)$ where $u_0 = ((1,y_0),\id)$. We therefore find that all of the conclusions of Lemma \ref{lem:line-cross-torus-sol} hold in this setting, in particular the solution $(u(t),g(t))$ does not degenerate, has $u(t) = ((1,y(t)),\id)$ with $y(t)$ constant in space, and the metric component satisfies the estimate
\begin{equation} \label{eqn:strip-metric-bound}
	\dwp (g(t),G_{\log y(t)}) \leq C (E(t)-1)^{\frac{1}{2}}.
\end{equation}

By restricting to initial data $(u_0,g_0)$ with energy $E(u_0,g_0)\leq f_0(0)$, 
e.g. by considering
\beq \label{eqn:initial-data}
(u_0,g_0) = (((1,y_0),\id),G_{\log y_0}), \quad y_0\geq 1,
\eeq 
we can ensure that $y(t) \geq 1$ along the flow, see \eqref{est:energy-equal-f}, and hence that all solutions of the flow  under consideration remain in a fixed compact subset of $(N_1,g_{N_1})/\Gamma$. Hence we can modify $(N_1,g_{N_1})/\Gamma$ away from this compact subset to obtain a closed target manifold $(N,g_N)$ without changing the solutions of the flow from these initial data.

We now show that these solutions have the winding properties claimed in our first main theorem.

\begin{proof}[Proof of Theorem \ref{thm:main}]
Let $(N,g_N)$ be the closed target as defined above and let $(u,g)$, $u=((1,y),\id)$ be a solution of the flow from initial data $(u_0,g_0)$ as in \eqref{eqn:initial-data}. 
We are free to view this flow as mapping into $(N_1,g_{N_1})/\Gamma$
or $(N_1,g_{N_1})$ as is convenient.
Because of \eqref{eqn:energy-decay} we can choose a sequence $t_j\to \infty$ so that 
\begin{equation}
\label{banana}
\norm{\partial_t(u,g)(t_j)}_{L^2(M,g(t_j))}\to 0\quad\text{ as }j\to \infty.
\end{equation}

We first claim that
$y(t_j)\to\infty$. If this were not true then $y(t_j)$ would have a bounded subsequence as $y(t)\geq 1$, so by passing to a further subsequence we could ensure that 
$y(t_j)\to \tilde y$ for some finite $\tilde y$.
In particular, the map $\tilde u:=((1,\tilde y),\id)$ would be a smooth limit of the maps $u(t_j)$.
Next, by \eqref{eqn:strip-metric-bound} we know that the metrics $g(t_j)$ must remain a bounded distance from $G_{\log(y(t_j))}\to G_{\log(\tilde y)}$,
and are hence contained in a compact subset of $\Mab$. Thus we may pass to a further subsequence and obtain smooth convergence of $g(t_j)$ to some limit metric $\tilde g$.
By \eqref{banana}, we can deduce that $(\tilde u,\tilde g)$ is a stationary point, 
which contradicts Remark \ref{rem:no-crit-point}, completing the proof of the claim.

For each $t\in [0,\infty)$, let $n(t)$ be the unique nonnegative integer so that $\hat y(t)= y(t)e^{-n(t)}\in [1,e)$. By the claim above, 
we have $n(t_j)\to\infty$
and, after passing to a subsequence, we may assume that 
$\hat y(t_j)\to \bar y$ for some $\bar y\in [1,e]$. As the map component of the flow can be represented as 
$$u(t)=((1,y(t)), \id)\sim ((e^{-n(t)}, \hat y(t)), \varphi^{-n(t)})$$
  we thus obtain that the pulled back maps 
$u(t_j)\circ \varphi^{n(t_j)}:T^2\to (N_1,g_{N_1})/\Gamma$ converge smoothly to $\bar u=((0,\bar y),\id)$.

According to \eqref{eqn:strip-metric-bound}  the pulled-back metrics $(\varphi^{n(t_j)})^*g(t_j)\in\Mab$  remain 
a bounded
distance from $(\varphi^{n(t_j)})^*G_{\log(y(t_j))}=(\varphi^{n(t_j)})^*G_{\log(\hat y(t_j))+n(t_j)}= G_{\log(\hat y(t_j))} 
\to G_{\log(\bar y)}$
and are hence contained in a compact subset of $\Mab$. After passing to a further
subsequence, we thus obtain that $(\varphi^{n(t_j)})^*g(t_j)\to \bar g$ for some $\bar g\in \Mab$. 

By \eqref{banana}, the obtained limit $(\bar u,\bar g)$ must be a 
stationary point of the flow. In particular $\bar u$ must be conformal and hence we must have $\bar g=G_{\log \bar y}$.
 As the energy decreases along the flow we conclude that
\beq \label{eqn:ice-cream}
\lim_{t\to \infty}E(t)=E(\bar u,\bar g)=f(0,\bar y)=1.
\eeq
As $E(t)\geq f(1,y(t))\geq 1$ we must therefore have that $f(1,y(t))\to 1$ and so 
$ y(t)\to \infty$ as $t\to \infty$. Combining \eqref{eqn:ice-cream} with \eqref{eqn:strip-metric-bound} also yields the claimed convergence of metrics \eqref{claim:main-thm-metric} for $z(t) = \log y(t)$, which tends to infinity as $t \to \infty$.

It remains to prove the second part of the theorem. Let 
$z \in [0,1)$ be any fixed number. As $z(t)\to\infty$ as $t \to \infty$ we can 
pick times $t_j^z \to \infty$ (for sufficiently large $j$)  such that $z(t_j^z) = z + j$. As above we conclude that $(\varphi^j)^*(u(t_j^z),g(t_j^z))\to (u_z, G_z)$ where $u_z = ((0,e^z),\id)$. We finally observe that for $z_1 \neq z_2$ these maps are minimal immersions with disjoint images and so parametrise different minimal surfaces in $(N,g_N)$. 
\end{proof}

\section{A \L{}ojasiewicz Inequality and Convergence of the Flow} \label{sect:3}

The key step in proving convergence of the flow 
that is uniform in time
will be establishing a suitable \L{}ojasiewicz-Simon inequality, recalling that the flow is the gradient flow of the functional $E$ on the set of equivalence classes of maps and metrics under the action of pulling back by diffeomorphisms homotopic to the identity. Such inequalities were first used in the study of infinite dimensional gradient flows by Simon \cite{Simon1983} and 
our principal task is to develop the flow theory to a point at which we can invoke his ideas.

We will use the structure of the Banach manifold $\M_{-1}^s$ of hyperbolic metrics with coefficients in the Sobolev space of $H^s$ as discussed for example in \cite{Tromba1992}. We write $\M_{c}^s$ for the space of metrics of constant curvature $c\in\{-1,0\}$ with coefficients in the Sobolev space $H^s$.
The main result we need for the proof of Theorem \ref{thm:analytic} can be formulated as follows.

\begin{thm} \label{thm:loj-ineq}
	Let $(N,g_N)$ be a closed real analytic manifold,  let $M$ be a closed oriented surface of genus $\ga\geq 1$, and let $(\ubar,\gbar)$, $\gbar\in\M_c$, be a critical point of the Dirichlet energy $E(u,g)$.
 Then for any $s>3$ there is a neighbourhood $\mathcal O$ of $(\ubar,\gbar)$ in $H^{s}(M,N)\times\M_{c}^s$, $\alpha \in (0,\half)$ and $C < \infty$ such that for any $(u,g) \in \mathcal{O}$ we have
	\begin{equation} \label{eqn:loj-ineq}
	\abs{E(u,g) - E(\ubar,\gbar)}^{1-\alpha} \leq C \left( \norm{\tau_g(u)}_{L^2(M,g)}^2 + \norm{ P_g(\Phi(u,g))}_{L^2(M,g)}^2 \right)^\half,
	\end{equation}
	where $P_g$ is the $L^2(M,g)$-orthogonal projection from the space of quadratic differentials to the space $\Hol(g)$ of holomorphic quadratic differentials. 
\end{thm}
The simpler, analogous result for $\ga=0$, where the final term of \eqref{eqn:loj-ineq} is always zero because all holomorphic quadratic differentials are trivial, is already included in the work of Simon \cite{Simon1983}.

If $(u,g)$ is a solution of the flow \eqref{eqn:flow} then 
the right-hand side of \eqref{eqn:loj-ineq} is controlled by the $L^2$-norm of the velocity of $(u,g)$, modulo a factor that is allowed to depend on $\eta$. 
Arguing as in the proof of \cite[Lemma 1]{Simon1983}, we can thus 
use Theorem \ref{thm:loj-ineq} to
control the $L^2$-length of any solution $(u,g)$ of the flow on intervals $I=(s_1,s_2)$ for which $(u,g)\vert_I$ is contained in $\mathcal O$ and for which $E(u,g)\geq E(\ubar,\gbar)$. To be more precise, setting $\Delta E(t) = E(u(t),g(t))-E(\ubar,\gbar)\geq 0$ 
	 and combining the \loj inequality with \eqref{eqn:energy-decay} and \eqref{eqn:flow}, 
gives
	\begin{equation*}
	\frac{d}{dt}(\Delta E(t))^\alpha = -\alpha \left(  \Delta E(t) \right)^{\alpha -1}\left(\norm{\partial_t u}_{L^2(M,g)}^2 + \eta^{-2}\norm{\partial_t g}_{L^2(M,g)}^2\right) \\
	\leq -C^{-1}
	\norm{\partial_t(u,g)}_{L^2(M,g)}
	\end{equation*}
for $C$ independent of $t$, so 
	\begin{equation} \label{eqn:length-energy-bd}
	\int_{s_1}^{s_2} \norm{\partial_t(u, g)}_{L^2(M, g)} dt \leq C
	(\Delta E(s_1))^\alpha.
	\end{equation}
We will give the proof of Theorem \ref{thm:loj-ineq} later and first explain how this implies our second main result, Theorem \ref{thm:analytic}, about the convergence of the flow. 
When $M = T^2$ the problem is simplified considerably by the fact that the metric part of the flow is constrained to lie in $\Mab$. We shall hence focus for now on the case of surfaces of genus $\gamma\geq 2$ and will outline the argument for the simpler case of domains of genus $1$ in Remark \ref{rem:conv-sphere-torus}. 

\begin{proof}[Proof of Theorem \ref{thm:analytic}]
	Let $(u,g)$ be as in the theorem. As there is no degeneration of the metric $g$ along the sequence of times $t_j \to \infty$ 
	we can pass to a subsequence (not relabelled) so that there are orientation preserving diffeomorphisms $\psi_j \colon M\to M$ for which 
	\begin{equation} \label{eqn:conv-u-g}
		\psi_j^*g(t_j)\to \gbar \text{ smoothly and } \psi_j^{*} u(t_j):= u(t_j)\circ \psi_j\weakto \ubar  \text{ weakly in }H^1,
	\end{equation}
where $(\ubar,\gbar)\in C^\infty(M,N)\times \M_{-1}$ is a critical point of $E$, see \cite{RT2015}. 
Although here we have only claimed weak $H^1$ and hence strong $L^2$ convergence, in the present situation the first hypothesis of \eqref{ass:map-metric-converg},
combined with the convergence of the metrics,
tells us that we have strong $H^1$ convergence.
This is all we need for now, but later in the proof, parabolic regularity theory will yield uniform $C^k$ bounds on  $\psi_j^{*} u(t_j)$, see \eqref{eqn:Ck-control-map} below, allowing us to deduce that also the convergence of the map component in \eqref{eqn:conv-u-g} is smooth. 
Note that because different smooth domain metrics all lead to the same notion of $H^s$ (or $C^k$) convergence, here and in the following there is no need to specify the domain metric when talking about convergence.

By the convergence of the metric in \eqref{eqn:conv-u-g}, we may drop a finite number of terms in $j$ in order to assume that 
\beq\label{factor2equiv}
\half\gbar\leq \psi_j^*g(t_j)\leq 2\gbar
\eeq
for each $j$. It will also be convenient to drop finitely many terms so that $t_j\geq 1$.

	Let   $s\in\N_{> 3}$ and let $\mathcal{O}$ be the neighbourhood of $(\ubar,\gbar)$ on which the \loj inequality \eqref{eqn:loj-ineq} holds. Let $V = \{h \in \M_{-1}^s \colon \norm{h - \gbar}_{H^s(M,\bar g)} < \sigma \}$ and $U = \{ w \in H^{s} \colon \norm{w - \ubar}_{H^{s}(M,\bar g)} < \sigma\}$ and assume that $\sigma>0$ is chosen small enough so that $U\times V \subset \mathcal{O}$ and so that the projection of $V$ onto 
moduli space lies in a compact subset.

For each $j$, if $\psi_j^*(u,g)(t_j)\in U\times V \subset \mathcal O$ we let $T_j >  t_j$ be the maximal time so that $\psi_j^*(u,g)(t)\in U\times V \subset \mathcal O$ for all $t \in [t_j,T_j)$; otherwise we set $T_j = t_j$. 
As $(\tilde u_j,\tilde g_j) := \psi_j^*(u,g)$ is also a solution of the flow \eqref{eqn:flow} we can apply \eqref{eqn:length-energy-bd} to conclude that 
for $C$ independent of $j$,
	\begin{equation} \label{eqn:small-length-j}
	\int_{t_j}^{T_j} \norm{\partial_t(u, g)}_{L^2(M, g)} dt = \int_{t_j}^{T_j} \norm{\partial_t(\tilde u_j, \tilde g_j)}_{L^2(M, \tilde g_j)} dt \leq C 
	(\Delta E(t_j))^\alpha\to 0
	\end{equation}	
	as $j\to \infty$ since $\Delta E(t_j)=E((u,g)(t_j))-E(\ubar,\gbar)\searrow 0$
	as a result of the strong $H^1$ convergence of $\tilde u_j$ and smooth convergence of $\tilde g_j$.

	The main part of the proof will be to show that for sufficiently large $j$ we have $T_j = \infty$. Once we have established this, we can then use \eqref{eqn:small-length-j} to deduce that $(u,g)$ converges to a critical point of $E$ as $t\to \infty$. While \eqref{eqn:small-length-j} on its own would only indicate $L^2$-convergence we will later be able to combine it  with parabolic regularity theory for the map and the properties of horizontal curves of hyperbolic metrics \cite{RT2018}
to conclude that the flow indeed converges smoothly as claimed.

We will have to ensure we have control of the metric and map on suitable intervals. As the metric component is given by a horizontal curve, the results of \cite{RT2018} yield appropriate $C^k$ control on the metric as long as we have a positive lower bound on the injectivity radius. Such a bound on $\inj(M,g)$ holds trivially for metrics in the $H^s$ neighbourhood $V$ of $\gbar$ as the projection of $V$ onto moduli space is assumed to be compact. Furthermore, as $\inf_j \inj(M,g(t_j))>0$, and by the well-known fact that
the evolution of the injectivity radius is controlled by the $L^2$-velocity of $g$, 
which is constrained by \eqref{cherry},
we can fix $\tau_0 \in (0,1]$ 
in a way that ensures that the injectivity radii of the surfaces $(M,g(t))$ are bounded away from zero  for $t \in [t_j - \tau_0,t_j+\tau_0]$, uniformly in $j$. 
Combined we thus find that there is some $\de_0>0$ so that 
\beq \label{est:inj-de-0}
\inj(M,g(t))\geq \de_0 \text{ for every } t\in I_j^{\tau_0}, \text{ for every } j\in \N,\eeq
where we  define 
\beq \label{def:Ij-tau} 
I_j^{\tau} := [t_j - \tau, \max(t_j+\tau,T_j) ) \text{ for } \tau>0.
\eeq
In addition we also have a uniform upper bound on the $L^2$-length of $g\vert_{I_j^{\tau_0}}$ of the form 
$C E(0)^\al + C \sqrt{E(0)\tau_0}\leq C$,  
for $C$ independent of $j$,
compare \eqref{eqn:small-length-j} and \eqref{cherry}. 
This together with \eqref{est:inj-de-0} allows us to split each $I_j^{\tau_0}$ into a fixed $j$-independent number of subintervals on which Lemma 3.2 of \cite{RT2018} is applicable. This yields that 
the metrics $\tilde g_j(t)$ are uniformly equivalent on $I_j^{\tau_0}$ 
i.e. comparable by a factor that is independent of $j$.
By \eqref{factor2equiv}, we find that there exists $C>0$ 
independent of $j$ so that 
\beq\label{metric_equiv}
C^{-1}\gbar\leq \tilde g_j(t)\leq C\gbar\qquad\text{ for every }t\in I_j^{\tau_0}.
\eeq
Lemma 3.2 of \cite{RT2018} also yields that
for every $k\in\N$, there exists $C<\infty$ independent of $j$ such that  
for every $t\in I_j^{\tau_0}$ we have
\begin{equation} \label{eqn:Ck-control-metric}
	\norm{\partial_t \tilde g_j(t)}_{C^k(M,\gbar)} \leq 
	C\norm{\partial_t \tilde g_j(t)}_{C^k(M,\tilde g_j(t_j))}
	\leq C	\norm{\partial_t \tilde g_j(t)}_{L^2(M,\tilde g_j(t))} 
\end{equation}
where the first inequality holds true as $\tilde g_j(t_j)$ converges smoothly to $\gbar$.

Having established this control on the metric component on the intervals $I_j^{\tau_0}$ we now turn to the analysis of  $\tilde u_j(t) = u(t)\circ \psi_j$ on these intervals. We recall that standard $H^2$-estimates, for example those found in section 3 of \cite{RT2019}, imply that there exist constants $\eps_0=\eps_0(N,g_N)>0$ and $C=C(N,g_N)$ so that for any hyperbolic metric $g$, any $H^2$ map $v:M \to N$ and any $r\in (0,\inj(M,g)]$ with 
\begin{equation} \label{eqn:no-concentration}
	\sup_{x \in M} 
	\half\int_{B^{g}_{r}(x)} \abs{dv}_{g}^2 dv_{g}
 \leq 2\eps_0
\end{equation}	
we have 
\begin{equation} \label{eqn:eps-reg}
	\int_{B^g_{r/2}} |\nabla_g^2 v|_g^2 + |d v|_g^4 dv_g \leq C\left(r^{-2}E(v;B^g_r) + \norm{\tau_g(v)}_{L^2(M,g)}^2\right).
\end{equation}


We will show that we can choose $r\in(0,\de_0]$, $\tau\in (0, \tau_0]$ and $j_0$ so that
\eqref{eqn:no-concentration} holds true 
for $g=g(t)$ and $v=u(t)$
on $I^{\tau}_j$ for every $j\geq j_0$. 
Letting $r_1>0$ be as in 
Remark \ref{rem:weaker-energy-ass}, this will follow from that remark, 
for appropriate 
$r\in (0,r_1]$ if we can control the  evolution of the local energy
not only forward in time but also backwards in time.

To do this, we first note that without loss of generality, by reducing $r_1>0$ if necessary, we may assume that $r_1\leq \de_0$.
Fix a smooth cut-off function $\si:\R\to [0,1]$ with $\si(\rho)=1$ for $\rho\leq r_1/2$ and $\si(\rho)=0$ for $\rho\geq 3r_1/4$.
For each $x\in M$, define $\vph\in C_c^\infty(B_{r_1}^{g(t_j)}(x))$ 
by $\vph(y):=\si(d_{g(t_j)}(y,x))$.
Appealling to the lower injectivity radius bound for $g(t)$ on $I_j^{\tau_0}$, and the holomorphicity of $\partial_t g$
we deduce that 
$\norm{\partial_tg}_{L^\infty(M,g)}\leq C\norm{\partial_t g}_{L^1(M,g)}\leq CE(0)$ (cf. the proof of \cite[Lemma 2.6]{RT2018}).
Consequently, 
Lemma 3.2 from \cite{RT2019}, combined also with \eqref{metric_equiv},
allows us to control the evolution of the local energy by
\beqa
\bigg|\frac{d}{dt}\left(\half\int \vph^2\abs{du}_{g}^2 dv_{g}\right)
+\int \vph^2|\tau_g(u)|^2dv_g \bigg|
&\leq
C\left(\int \vph^2|\tau_g(u)|^2dv_g\right)^\half +C\\
&\leq C+\int \vph^2|\tau_g(u)|^2dv_g
\eeqa
for each $t\in I_j^{\tau_0}$, where $C$ is independent of 
$j$, $t$ and $x$.
Therefore, for $t\in I_j^{\tau_0}$ with $t<t_j$ we obtain that  
\beqa
\label{est:cake24} 
\half\int_{B^{g(t_j)}_{r_1/2}(x)} \abs{du(t)}_{g(t)}^2 dv_{g(t)} 
& \leq \half \int_{B^{g(t_j)}_{r_1}(x)} \abs{du(t_j)}_{g(t_j)}^2 dv_{g(t_j)} \\
& \quad +C
(t_j-t)+2\int_{t}^{t_j} \norm{\tau_g(u)}_{L^2(M,g)}^2 dt',\eeqa
while for $t\in I_j^{\tau_0}$ with $t>t_j$ we have
\beqs  
\half\int_{B^{g(t_j)}_{r_1/2}(x)} \abs{du(t)}_{g(t)}^2 dv_{g(t)} 
\leq \half \int_{B^{g(t_j)}_{r_1}(x)} \abs{du(t_j)}_{g(t_j)}^2 dv_{g(t_j)} 
+C
(t-t_j),
\eeqs
see \cite{Struwe1985}, \cite{Topping04}, 
for the analogous results for harmonic map flow. 
As \eqref{eqn:energy-decay} implies that $t\mapsto \norm{\tau_g(u)}_{L^2}^2$ is integrable in time, 
 the tension term in \eqref{est:cake24}
is less than $\frac{\eps_0}{2}$ 
for all sufficiently large $j$ and any $t\in I_j^{\tau_0}$.
By \eqref{metric_equiv}, the metrics $g(t)$ for $t\in I_j^{\tau_0}$ 
are equivalent, by a factor independent of $j$,
so setting 
$r:=C^{-1}r_1\in (0,\de_0)$ 
for a suitably large  $C> 1$, 
independent of $j$, 
ensures that 
$B^{g(t)}_r(x)\subset B^{g(t_j)}_{r_1/2}(x)$ for every $j\in\N$, $t\in I_j^{\tau_0}$ and $x\in M$.

We can thus choose $\tau\in (0,\tau_0]$
and $j_0$ so that 
\eqref{eqn:no-concentration} holds for $v=u(t)$ and $g=g(t)$ on the intervals $[t_j - \tau,t_j + \tau]$, $j\geq j_0$. Of course, \eqref{eqn:no-concentration} also holds trivially for maps and metrics in the $H^{s}$ neighbourhood $\mathcal O$ of $(\ubar,\gbar)$ if $\si>0$ is initially chosen small enough, after reducing $r$ if necessary. 
Therefore we obtain that  \eqref{eqn:Ck-control-metric} and, 
for $(v,g)=(u(t),g(t))$,
\eqref{eqn:no-concentration} and hence \eqref{eqn:eps-reg} hold on the interval $I_j^{\tau}$ 
for $j\geq j_0$.
Standard parabolic theory, again see for example \cite{RT2019}, now yields that for every $k\in \N$ there exists a constant $C$ independent of $j$ 
so that 
\begin{equation} \label{eqn:Ck-control-map}
 \norm{\tilde u_j(t)}_{C^k(M,\gbar)} 
 \leq C \quad\text{  for all  } t\in [t_j,\max(t_j+\tau,T_j)).
\end{equation}
As a first application this implies that $\tilde u_j(t_j)$  converges to  $\ubar$ not only in $H^1$ but in $C^k$ for every $k$ as claimed earlier. 

Let now $\eps>0$ be a constant to be determined below, 
independently of $j$. Then \eqref{eqn:small-length-j} and the smooth convergence of maps and metrics ensures that 
	\begin{equation} \label{eqn:small-u-g-tj}
	\norm{\tilde u_j(t_j) - \ubar}_{H^{s}(M,\gbar)} + \norm{\tilde g_j(t_j) - \gbar}_{C^{s}(M,\gbar)} +  \int_{t_j}^{T_j} \norm{\partial_t(\tilde u_j, \tilde g_j)}_{L^2(M, \tilde g_j)} dt \leq \eps,
	\end{equation}
for all sufficiently large $j$ depending, in particular, on $\ep$. We claim that 
if $\eps>0$ is chosen small enough then for every $j\geq j_0$ for which \eqref{eqn:small-u-g-tj} holds we have that both
$\norm{\tilde g_j(t)- \gbar}_{H^s(M,\gbar)} < \sigma/2$ and $\norm{\tilde u_j(t)-\ubar}_{H^{s}(M,\gbar)} < \sigma/2$ for every $t \in [t_j,T_j)$ and hence $T_j = \infty$. First we deal with the metric component; we have for every $t \in [t_j,T_j)$
\begin{equation} \label{eqn:Hs-norm-small-metric}
\begin{split}
	\norm{\tilde g_j(t) - \gbar}_{C^{s}(M,\gbar)} &\leq \norm{\tilde g_j(t_j) - \gbar}_{C^{s}(M,\gbar)} + \int_{t_j}^{T_j} \norm{\partial_t \tilde g_j }_{C^{s}(M,\gbar)}dt  \\
	&\leq \norm{\tilde g_j(t_j) - \gbar}_{C^{s}(M,\gbar)} + C \int_{t_j}^{T_j} \norm{\partial_t \tilde g_j }_{L^2(M,\tilde g_j(t))}dt \leq \eps + C\eps
\end{split}
\end{equation}
for $C$ independent of $j$ and $\ep$,
where we have used  \eqref{eqn:Ck-control-metric} and \eqref{eqn:small-u-g-tj}. This establishes the necessary claim on the metric component if $\eps>0$ is chosen small enough.

To deal with the map component we first use \eqref{eqn:small-u-g-tj} 
and \eqref{metric_equiv} to conclude that for every $t \in [t_j,T_j)$
\beqa
\label{eqn:L2-norm-small}
\norm{\tilde u_j(t) - \ubar}_{L^2(M,\bar g)} &\leq 
\norm{\tilde u_j(t_j) - \ubar}_{L^2(M,\bar g)}	
+ \int_{t_j}^{t} \norm{\partial_t \tilde u_j }_{L^2(M,\gbar)}dt\\
&\leq 
\norm{\tilde u_j(t_j) - \ubar}_{H^{s}(M,\gbar)}	
+ C \int_{t_j}^{t} \norm{\partial_t \tilde u_j }_{L^2(M,\tilde g_j(t))}dt\\
&\leq C\eps,
\eeqa
for $C$ independent of $j$ and $\ep$.
At the same time \eqref{eqn:Ck-control-map}  implies the uniform bound
$\norm{\tilde u_j(t)-\ubar}_{C^{s+1}(M,\gbar)}\leq C_1$
on $[t_j,T_j)$, where $C_1$ is independent of $j$ and $\ep$.
Because $C^{s+1}$ embeds compactly into $H^{s}$, which in turn embeds continuously into $L^2$, by Ehrling's lemma we know that for every $\de>0$ there exists a number $C$
so that for every $w\in C^{s+1}(M,N)$ we have
$$\norm{w}_{H^{s}(M,\gbar)}\leq \de \norm{w}_{C^{s+1}(M,\gbar)} +C \norm{w}_{L^2(M,\gbar)}.$$
Applied for $\de=\frac{\si}{4C_1}$ this allows us to conclude that on $[t_j,T_j)$
$$\norm{\tilde u_j(t) - \ubar}_{H^{s}(M,\gbar)}\leq \frac{\si}{4}+C\eps\leq \frac{\si}{2}$$
where the last inequality holds provided $\eps>0$ is initially chosen small enough. 
This concludes the proof of our claim that $T_j=\infty$ for every $j\geq j_0$ such that \eqref{eqn:small-u-g-tj} holds for an $\ep>0$ that can now be considered fixed.

Let us  fix  $J \in \N$, $J\geq j_0$, large enough so that \eqref{eqn:small-u-g-tj} holds with this $\eps$ for $j=J$ 
and hence so that 
$T_J=\infty$. 
Thus \eqref{eqn:length-energy-bd} can be applied on all of $[t_{J},\infty)$ allowing us to conclude that
\begin{equation*}
 	\int_{t}^\infty \norm{\partial_t(u, g)}_{L^2(M, g)} dt \leq C (\Delta E(t))^\alpha\to 0 \text{ as } t\to \infty.
\end{equation*} 
Using this we will be able to show that the original flow $(u, g)(t)$ converges in $L^2$ 
as $t\to\infty$ without having to restrict to a sequence of times $t_i\to\infty$
and without having to pull back by diffeomorphisms.
Indeed, because 
\beq
\label{semolina}
\int_{t}^\infty \norm{\partial_t(\tilde u_J, \tilde g_J)}_{L^2(M, \tilde g_J)} dt
=\int_{t}^\infty \norm{\partial_t(u, g)}_{L^2(M, g)} dt
\leq C (\Delta E(t))^\alpha\to 0
\eeq
we can compute
\beqa
\int_{t}^\infty \norm{\partial_t(u, g)}_{L^2(M, (\psi_J^{-1})^*\gbar)} dt 
&=\int_{t}^\infty \norm{\partial_t(\tilde u_J, \tilde g_J)}_{L^2(M, \gbar)} dt \\
& \stackrel{\eqref{metric_equiv}}{\leq}
C\int_{t}^\infty \norm{\partial_t(\tilde u_J, \tilde g_J)}_{L^2(M, \tilde g_J)} dt \\
&\to 0,
\eeqa
which implies that $(u,g)(t)$ converges in $L^2$ to some limit $(u_\infty,g_\infty)$ as $t\to \infty$.

Moreover, as the curve of metrics is horizontal and as we have a uniform lower bound on the injectivity radius on all of $[t_J,\infty)$ we have
\begin{equation}
\int_{t}^\infty \norm{\partial_tg}_{C^k(M,(\psi_J^{-1})^*\bar g)} dt 
=
\int_{t}^\infty \norm{\partial_t \tilde g_J(t)}_{C^k(M,\bar g)} dt 
\stackrel{\eqref{eqn:Ck-control-metric}}{\leq}
C\int_{t}^\infty 
\norm{\partial_t \tilde g_J(t)}_{L^2(M,\tilde g_J(t))} dt\to 0
\end{equation}
by \eqref{semolina}, which can be integrated to find that the 
metrics converge to $g_\infty$ \emph{smoothly} as $t\to\infty$.
In addition, \eqref{eqn:Ck-control-map} yields uniform bounds on $u(t)$ in $C^k(M,(\psi_J^{-1})^*\gbar)$, $t\in [t_J,\infty)$, 
so also the map $u$ converges smoothly to its limit $u_\infty$.

Now that we have the smooth convergence, we see readily that the limit $(u_\infty,g_\infty)$
is a critical point of $E$.
\end{proof}

\begin{rem}\label{rem:conv-sphere-torus}
For surfaces of genus $1$ this proof still applies but can be simplified. 
As observed previously we may assume that $g(0)\in \Mab$ and hence that the flow of metrics is constrained to $\Mab$. 
Also we may choose the diffeomorphisms $\psi_j$ in the above proof so that $\psi_j^*\Mab=\Mab$  and hence $\tilde g_j(t)\in \Mab$ 
for any $t$ and
 $\bar g\in \Mab$. 
 Combining the convergence of the 
  $\tilde g_j(t_j)$ with the 
uniform upper bound on the $L^2$-length of $\tilde g_j\vert_{I_j^{\tau_0}}$, say for $\tau_0:=1$, and the completeness of $\Mab$ with respect to  the Weil-Petersson distance yields that 
 the metrics  $\tilde g_j(t)$, $t\in I_j^{\tau_0}$, are contained in a ($j$ independent) compact subset of $\Mab$. 
Hence \eqref{est:inj-de-0} and  \eqref{metric_equiv} still hold and the $C^k$ estimate \eqref{eqn:Ck-control-metric} trivially follows from the explicit form of the metrics $g_{a,b}$. 
For the rest of the proof we can then argue exactly as above. 
\end{rem}

We now turn to the proof of Theorem \ref{thm:loj-ineq} where we again focus on the case of surfaces of genus $\gamma\geq 2$, while the case of tori will be discussed later in Remark \ref{rem:Loj-sphere-torus}. 

For surfaces of genus $\gamma\geq 2$ 
we will first consider the special case of (not necessarily hyperbolic) metrics that are contained in a finite-dimensional affine space, then conclude that the desired inequality holds for metrics in a 
finite-dimensional slice of hyperbolic metrics as considered e.g. in \cite{Tromba1992} and finally use the structure of the space of hyperbolic metrics, in particular the so-called \emph{Slice Theorem}, to lift the inequality to a neighbourhood of $\bar g$ in $\M^s_{-1}$. 

\begin{prop} \label{prop:loj-affine}
	Let $(N,g_N)$ be a closed real analytic manifold and let $M$ be a closed oriented surface of genus $\gamma \geq 2$. Fix $s > 3$ and a smooth critical point $(\ubar,\gbar)$ 
	of the Dirichlet energy $E(u,g)$. Then there exist 
numbers $\alpha \in (0,\half )$ and $C < \infty$, a neighbourhood $\hat U$ of $\ubar$ in $H^{s}(M,N)$ and a neighbourhood $\hat V$ of $0$ in $\Rea(\Hol(M,\bar g))$,
with the property that each element of $\gbar+\hat V$ is a metric, 
such that for any $u \in \hat U$ and any  $g \in \gbar + \hat V$  we have
	\begin{equation} \label{eqn:loj-ineq-affine}
	\abs{E(u,g) - E(\ubar,\gbar)}^{1-\alpha} \leq C \left( \norm{\tau_g(u)}_{L^2(M,g)}^2 + \norm{ \Phi(u,g)}_{L^1(M,g)}^2 \right)^\half.
\end{equation}
\end{prop}

\begin{proof}
Given a closed analytic manifold $(N,g_N)$ we can first use Nash's embedding theorem to isometrically embed $(N,g_N)$ into a suitable Euclidean space $(\R^n,g_{\R^n})$ in an analytic manner. We can then modify the metric on $\R^n$ in a tubular neighbourhood of $N$, as described in \cite[Lemma 4.1.2]{Helein2002}, to obtain a new metric $h$ on $\R^n$, which is analytic in a neighbourhood of $N$, such that $N$ is a totally geodesic submanifold of $(\R^n,h)$ and so that at each point $p \in N$ we have $h(p) = g_{\R^n}(p)$. 
If we view any  map $u:(M,g)\to (N,g_N)$ as a map 
$u:(M,g)\to (\R^n,h)$ using this embedding, then the Hopf-differential remains unchanged and, as $N$ is totally geodesic, the tension fields can be identified.

The claim of the proposition hence follows immediately if we prove the inequality \eqref{eqn:loj-ineq-affine} for maps in an $H^s$ neighbourhood $\hat U$ of $\bar u$ in the larger space of maps from $(M,g)$ to $(\R^n,h)$.
From now on we hence consider the Dirichlet energy of maps $u:M\to (\R^n,h)$. 

We also identify the $6(\gamma -1)$ dimensional affine space $\gbar+\Rea(\Hol(\gbar))$ with $\R^{6(\gamma-1)}$ by fixing an 
$L^2(M,\gbar)$-orthonormal basis 
$\{k_i\}_{i = 1}^{6(\gamma - 1)}$  of $\Rea (\mathcal{H}(\gbar))$ and 
setting 
$\gmu = \gbar + \sum_{i=1}^{6(\gamma -1)}\mu_ik_i$. In the following we restrict $\mu$ to a neighbourhood $V$ of $0 \in \R^{6(\gamma -1)}$ so that $\half\bar g \leq \gmu \leq 2\bar g$
for all $\mu \in V$.

We now define
$ F\colon H^{s}(M,\R^n)\times V \to  \R $ by 
$$ F(u,\mu):= E(u:(M,\gmu)\to (\R^n,h)).
$$
We want to consider the gradient of $F$ with respect to the fixed inner product 
\beq \label{def:inner-product}
\langle (v,\mu), (\tilde v,\tilde \mu)\rangle = \lan v, \tilde v\ran_{L^2((M,\gbar),(\R^n,g_{\R^n}))} + \lan \mu,\tilde \mu \ran_{\R^{6(\gamma -1)}},
\eeq
i.e. 
$\nabla F(u,\mu) =: (\graduF(u,\mu), \gradmuF(u,\mu))$ satisfies
\beq \label{eqn:gradF}
	\frac{d}{d\eps}\bigg|_{\eps =0}F(u+\eps v,\mu+\eps\xi) = \langle \nabla F(u,\mu), (v,\xi)\rangle
	\eeq
for every $(v,\xi) \in H^{s}(M,\R^n)\times \R^{6(\gamma -1)}$.

As $\ddeps F(u+\eps v, \mu) = -\int_M \langle \tau_{\hat g(\mu)}(u), v\rangle_{h \circ u} \, dv_{\hat g(\mu)}$ we have that 
\begin{equation}\label{eq:grad-F-u}
	\nabla_u F(u,\mu) = -\sum_{i,j=1}^n(h\circ u)_{ij}(\tau_{\hat g(\mu)}(u))_j \psi(\mu)e_i
\end{equation}
where $(\tau_{\hat g(\mu)}(u))_i$ are the components of $\tau_{\hat g(\mu)}(u)$ with respect to the standard basis $\{e_i\}$ of $\R^n$ and $\psi(\mu) \colon M \to (0,\infty)$ is characterised by the relation $dv_{\hat g(\mu)} = \psi(\mu)dv_{\bar g}$ between the volume forms.

As the $L^2(M,\gmu)$-gradient at $g=\gmu$
of $g\mapsto E(u,g)$ 
on the space of all metrics is given by 
$-\tfrac14 \Rea(\Phi(u,\gmu))=-\frac14\big[ u^*h-\half\tr_{\gmu}(u^*h)\gmu\big]$ 
(see, e.g. \cite[Theorem 3.1.3]{Tromba1992})
we can furthermore see that $\na_\mu F=(\partial_{\mu_j} F)_{j=1,\ldots, 6(\gamma-1)}$ is given by
\begin{equation}\label{eq:grad-F-mu}
	\partial_{\mu_j} F(u, \mu) =-\tfrac14 \langle \Rea(\Phi(u,\gmu)), k_j\rangle_{L^2(M,\gmu)}.
\end{equation}

Hence
$$\nabla F \colon H^{s}(M,\R^n) \times V \to H^{s-2}(M,\R^n)\times \R^{6(\gamma-1)}$$
is well defined. We now collect some properties of $F$ and its gradient which will allow us to prove a \loj inequality for $F$.

Keeping in mind that $\grad F(\bar u,0)=0$, we consider the  linearisation 
 $$L \colon H^{s}(M,\R^n)\times \R^{6(\gamma -1)} \to H^{s-2}(M,\R^n)\times \R^{6(\gamma-1)}$$
 of $\na F$ at $(\ubar,0)$.
Note that $L$ is the Hessian of $F$ and so, as observed in \cite{Simon1983}, is formally self-adjoint in the sense that $\lan L (v,\mu), (\tilde v,\tilde \mu)\ran=\lan (v,\mu), L (\tilde v,\tilde \mu)\ran$ for the inner product defined in \eqref{def:inner-product}. 
In practice, this will follow from a more abstract machinery that will be invoked later.

We recall that for $p \in N$ we have $h(p) = g_{\R^n}$, in particular $h \circ \bar u = g_{\R^n}$, that $\psi=1$ at $\mu =0$ and that $\tau_{\bar g}(\bar u)=0$. 
Thus the linearisation of 
$v\mapsto \graduF(\bar u+v,0)$
at $v=0$ 
agrees with the linearisation of $v\mapsto -\tau_{\bar g}(\bar u+v)$
at $v=0$, and is thus Fredholm of index 0 from $H^s(M,\R^n)$ to $H^{s-2}(M,\R^n)$
since it is a linear second-order elliptic differential operator with smooth coefficients that can be written as $-\Delta_{\bar g}$ plus 
lower order terms that constitute a compact perturbation.
At the same time, the linearisations of $\mu \mapsto \na F(\ubar,\mu)$ at 
$\mu=0$ and of $v\mapsto \gradmuF(\bar u+v,0)$ at $v=0$
are bounded linear operators with finite dimensional domain respectively range and are thus compact. Combined we hence obtain that $L$ itself is Fredholm of index $0$. 

Similarly, we consider the linearisation $L_{(u,\mu)}$ of $\nabla F$ around different points $(u,\mu) \in H^s(M,\R^n) \times V$. For $k \in \N$, $k\geq 2$, we denote by $\mathcal B_k$ the space of bounded linear operators $H^k(M,\R^n) \times \R^{6(\gamma-1)} \to H^{k-2}(M,\R^n) \times \R^{6(\gamma-1)}$ and by $\norm{\cdot}_{\mathcal B_k}$ the operator norm on this space. We observe from the formulae \eqref{eq:grad-F-u}, \eqref{eq:grad-F-mu}
and the density of $H^s$ in $H^2$ that
the linearisation $L_{(u,\mu)}\in {\mathcal B_s}$ has a (unique) 
extension to an element of $\mathcal B_2$ and the map $(H^s(M,\R^n) \times V, \norm{\cdot}_{H^s\times \R^{6(\gamma-1)}}) \to (\mathcal B_2, \norm{\cdot}_{\mathcal B_2})$ taking $(u,\mu) \mapsto L_{(u,\mu)}$ is continuous. Furthermore, 
the extension of $L$ to $\mathcal B_2$ is Fredholm with index 0.

We choose an $H^s(M,\R^n)$ neighbourhood $U$ of $\bar u$ so that the image of each map $u\in U$ is contained in the neighbourhood of $N\subset \R^n$ where the metric $h$ is analytic and hence so that the functional $F$ and its gradient $\nabla F$ are analytic on $U\times V$.

We can now apply the classical argument of Simon \cite{Simon1983}  to obtain a \loj inequality for $F$, as $F$ and its gradient $\nabla F$ have the required properties 
 that allow Simon's argument to go through, see \cite{LojNotes} and
 \cite[\S 2.4]{Huang2006}. 
Namely we have found
 an $H^s(M,\R^n)$ neighbourhood $U$ of $\bar u$ and a neighbourhood $V$ of $0 \in \R^{6(\gamma -1)}$ such that
 \begin{enumerate}
	\item The functional $F \colon U \times V \to \R$ and its gradient $\nabla F \colon U \times V \to H^{s-2}(M,\R^n)\times \R^{6(\gamma-1)}$ are both analytic.
	\item  For each $(u,\mu) \in U \times V$ the linearisation $L_{(u,\mu)} $ of $\nabla F$ around $(u,\mu)$ is an element of $\mathcal{B}_s$ and has an extension to an element of $\mathcal B_2$. Furthermore this map $(U \times V, \norm{\cdot}_{H^s\times \R^{6(\gamma-1)}}) \to (\mathcal B_2, \norm{\cdot}_{\mathcal B_2})$ is continuous. 
		\item The linearisation $L=L_{(\bar u,0)}\in \mathcal{B}_s$ of $\na F$ around the critical point $(\bar u, 0)$ is Fredholm with index $0$ and its extension to $\mathcal B_2$ is also Fredholm with index 0. 
\end{enumerate} 
We consequently obtain the following \loj inequality for $F$. 
There exist $\alpha \in (0,\frac{1}{2})$, $\si>0$ and $C < \infty$
such that we have the inclusions
$\hat U:=\{u\in H^s(M,\R^n): \norm{u-\ubar}_{H^s}< \sigma\}\subset U$ and
$\{\mu \in \R^{6(\gamma-1)}: \abs{\mu}< \sigma\}\subset V$, 
and we have the estimate
	\begin{equation} \label{eqn:est-F}
	|F(u,\mu) - F(\ubar,0)|^{1-\alpha} \leq C \norm{\nabla F(u,\mu)}
	\end{equation}
for all 
$u \in \hat U$ and $\mu \in \R^{6(\gamma -1)}$ with 
$|\mu| < \sigma$, where $\norm{\cdot}$ is the norm induced by the inner product \eqref{def:inner-product}.

We finally set 
$\hat V=\{h\in \Rea(\Hol(\gbar)): \norm{h}_{L^2(M,\gbar)}<\si\}$. To explain how \eqref{eqn:est-F} implies \eqref{eqn:loj-ineq-affine} we now need to relate $\na F$ to the quantities appearing on the right-hand side of \eqref{eqn:loj-ineq-affine}. 

The formula \eqref{eq:grad-F-u} for $\na_uF$, combined with the equivalence of the metrics $\gbar$ and $\gmu$ (and boundedness of $h$) immediately implies that
\beqa 
\norm{\graduF (u,\mu)}_{L^2(M,\gbar)}
\leq C\norm{\tau_{\gmu}(u)}_{L^2(M,\gmu)}. 
\eeqa
Similarly, 
we can use \eqref{eq:grad-F-mu} to bound 
	\begin{align*}
		\babs{\partial_{\mu_j} F(u, \mu) }&=\tfrac14\babs{ \langle \Rea(\Phi(u,\gmu)), k_j\rangle_{L^2(M,\gmu)}}  \leq \tfrac14 \norm{k_j}_{L^\infty(M,\gmu)}\norm{ \Rea(\Phi(u,\gmu)}_{L^1(M,\gmu)}\\
	&\leq C \norm{k_j}_{L^\infty(M,\gbar)}\norm{\Phi(u,\gmu)}_{L^1(M,\gmu)}\leq C \norm{\Phi(u,\gmu)}_{L^1(M,\gmu)}
	\end{align*}
where we use the equivalence of the metrics $\gbar$ and $\gmu$ in the penultimate step, while the last step follows as the $L^2(M,\gbar)$-norm is equivalent to the $L^\infty(M,\gbar)$-norm
on the finite-dimensional space $\Rea(\Hol(\gbar))$. 
Hence $\norm{\na F(u,\mu)} 
\leq C ( \norm{\tau_\gmu(u)}_{L^2(M,\gmu)}^2 + \norm{ \Phi(u,\gmu)}_{L^1(M,\gmu)}^2 )^\half$ and the claim of the proposition follows from \eqref{eqn:est-F}.
\end{proof}

We will now expand on the remarks given before Proposition \ref{prop:loj-affine} to explain how, for surfaces of genus $\gamma\geq 2$, Proposition \ref{prop:loj-affine} 
implies that the \loj inequality \eqref{eqn:loj-ineq} holds on a $H^{s}$ neighbourhood of $(\ubar, \gbar)$ as claimed in Theorem \ref{thm:loj-ineq}.

\begin{proof}[Proof of Theorem \ref{thm:loj-ineq}]
	Following \cite{Tromba1992} we consider the slice of hyperbolic metrics 
	\begin{equation}
		\mathcal{S} = \left\{ \lambda(k)(\gbar + k) \colon k \in \hat V  \right\}
	\end{equation}
	associated with the neighbourhood $\hat V$ obtained in Proposition \ref{prop:loj-affine} above, which we can assume is chosen small enough so that all metrics $g$ in $\bar g+\hat V$ satisfy $\half \bar g\leq g\leq 2\bar g$. 
 Here  $\lambda \colon \hat V \to C^\infty(M,(0,\infty))$ 
is uniquely defined by the condition that $\lambda(k)(\gbar + k)$ has constant curvature $-1$. We note that the functions $\la(k)$ are uniformly bounded 
for $k\in \hat V$ and thus that $\norm{\tau_{\bar g+k}(u)}_{L^2(M,\bar g+k)}\leq C\norm{\tau_{\la(k)(\bar g+k)}(u)}_{L^2(M,\la(k)(\bar g+k))}$, while  $\norm{\Phi(u,g)}_{L^1(M,g)}$ is invariant under conformal changes.
As the energy is conformally invariant we immediately deduce from Proposition \ref{prop:loj-affine} that there exists $C>0$ so that 	
\beq \label{est:chocolate}
\abs{E(\hat u,\hat g) - E(\ubar,\gbar)}^{1-\alpha} \leq C \left( \norm{\tau_{\hat g}(\hat u)}_{L^2(M,\hat g)}^2 + \norm{ \Phi(\hat u,\hat g)}_{L^1(M,\hat g)}^2 \right)^\half\text{ for } (\hat u,\hat g)\in \hat U\times \mathcal{S}
\eeq
where $\hat U$ is the $H^s$-neighbourhood of $\ubar$ obtained in Proposition \ref{prop:loj-affine}. 

The Poincar\'e inequality for quadratic differentials on hyperbolic surfaces, see \cite[Lemma 2.1]{RT2016} or \cite[Theorem 1.1]{RT2015}, implies that
	\begin{align*}
		\norm{\Phi(\hat u,\hat g) - P_{\hat g}(\Phi(\hat u,\hat g))}_{L^1(M,\hat g)} &\leq C\norm{\bar\partial \Phi(\hat u,\hat g)}_{L^1(M,\hat g)} \leq C\norm{\tau_{\hat g}(\hat u)}_{L^2(M,\hat g)}  .
	\end{align*}
As the area of $(M,\hat g)$ is determined by the genus of $M$ we can thus bound 
$$\norm{\Phi(\hat u,\hat g)}_{L^1(M,\hat g)}\leq C \norm{P_{\hat g}(\Phi(\hat u,\hat g))}_{L^2(M,\hat g)}+C\norm{\tau_{\hat g}(\hat u)}_{L^2(M,\hat g)}.$$
Inserting this into \eqref{est:chocolate} yields the claimed inequality \eqref{eqn:loj-ineq} in the special case of maps $\hat u\in \hat U$ and metrics $\hat g\in \mathcal{S}$.

As \eqref{eqn:loj-ineq} is invariant under pullback by diffeomorphisms we thus
know that the inequality \eqref{eqn:loj-ineq} holds for any pair $(u,g) = f^*(\hat u,\hat g)$ that is obtained by pulling back an element $(\hat u, \hat g) \in \hat U\times \mathcal{S}$ by an arbitrary $H^{s+1}$ diffeomorphism. 
Hence the theorem follows once we prove that there is a neighbourhood $\mathcal{O}$ of $(\ubar,\gbar)$ in $H^{s}\times \M_{-1}^s$ so that every pair in $\mathcal O$ can be written in this form.

Let $\mathcal{D}_0^{s+1}$ be the set of $H^{s+1}$ diffeomorphisms of $M$ that are homotopic to the identity. 
We recall that  the map $H^{s}\times \mathcal{D}_0^{s+1}\to H^s$ given by $(u,f) \mapsto u\circ f^{-1}$ is continuous, see for example Theorem 1.2 of \cite{Inci2013}. We can hence 
choose a neighbourhood $Q$ of the identity in $\mathcal{D}_0^{s+1}$ and a neighbourhood $U$ of $\ubar$ in $H^{s}$ so that $u\circ f^{-1}$ lies in the $H^s$ neighbourhood $\hat U$ for every $f\in Q$ and every $u\in U$. 

The slice theorem, see for example \cite{Tromba1992}, ensures that, after possibly reducing $Q$ and $\mathcal{S}$,
\begin{equation} \label{eqn:slice-map}
		\Xi : \mathcal{S}\times Q \to W, \quad (\hat g, f) \mapsto f^*\hat g
	\end{equation}
is a diffeomorphism onto a neighbourhood $W$ of $\gbar$ in $\M_{-1}^s$. We hence conclude that \eqref{eqn:loj-ineq} holds on the $H^{s}\times \M_{-1}^s$ neighbourhood $\mathcal{O} := U\times W$ of $(\ubar,\gbar)$ as claimed. 
\end{proof}

\begin{rem} \label{rem:Loj-sphere-torus}
For genus one surfaces we can modify the above proof of Theorem \ref{thm:loj-ineq} as follows. 
As we can pull back any given critical point $(\bar u,\bar g)$, $\bar g\in \M_0$, of $E$ by a smooth diffeomorphism to obtain a critical point of $E$ for which the metric component is in $\M^*$, it suffices to consider $(\bar u, \bar g)$ with $\bar g\in \M^*$. 
By considering the function $F: H^{s}(M,\R^n)\times \Hyp  \to \R$ given by $F(u,(a,b))= E(u:(T^2,g_{a,b})\to (\R^n,h))$ 
and arguing as in the proof of Proposition \ref{prop:loj-affine} for higher genus surfaces
we obtain the analogue to Proposition \ref{prop:loj-affine}, now for metrics in a $\Mab$ neighbourhood of $\bar g\in \Mab$, which we may assume to be contained in a compact set $K$.

To obtain Theorem \ref{thm:loj-ineq}, we note that even though the Poincar\'e estimate for quadratic differentials is not uniform for flat tori, it is still valid for metrics in the compact set $K$, now with a constant that also depends on $K$. 
Hence we can argue as above,
replacing the slice theorem of \cite{Tromba1992} 
by the fact that
the map $(g,f)\mapsto f^*g$
is a diffeomorphism onto a neighbourhood 
of $\bar g$ in $\M_0^s$, when considered for $g$ ranging over a small neighbourhood of $\bar g$ in $\Mab$,
and $f$ ranging over a small neighbourhood of the identity in the set of diffeomorphisms 
$\{f \in \mathcal{D}_0^{s+1}  :\ f(p^*)=p^*\}$ that keep some given point $p^*\in T^2$ fixed.

\end{rem}

\emph{Acknowledgements:} 
The first author was supported by EPSRC grant number EP/L015811/1.
The third author was supported by 
EPSRC grant number EP/K00865X/1.

{\sc JK \& MR: 
Mathematical Institute, University of Oxford, Oxford, OX2 6GG, UK}

{\sc PT: Mathematics Institute, University of Warwick, Coventry,
CV4 7AL, UK}

\end{document}